\documentclass[11pt,reqno]{amsart}

\usepackage[usenames]{color}
\usepackage{graphicx}
\usepackage{amscd}
\usepackage[colorlinks=true,
linkcolor=webgreen,
filecolor=webbrown,
citecolor=webgreen]{hyperref}
\usepackage{cite}
\definecolor{webgreen}{rgb}{0,.5,0}
\definecolor{webbrown}{rgb}{.6,0,0}
\usepackage{amsthm}
\usepackage{fullpage}
\usepackage{enumerate}

\usepackage{graphics,amsmath,amssymb}

\usepackage{amsfonts}
\usepackage{latexsym}
\usepackage{url}

\newcommand{\seqnum}[1]{\href{http://oeis.org/#1}{\underline{#1}}}
\theoremstyle{plain}
\newtheorem{theorem}{Theorem}
\newtheorem{corollary}[theorem]{Corollary}
\newtheorem{lemma}[theorem]{Lemma}
\newtheorem{proposition}[theorem]{Proposition}

\theoremstyle{remark}

\begin{document}
\begin{center}
	\title[The Geometry of some Fibonacci Identities in the Hosoya Triangle]
	{The Geometry of some Fibonacci Identities in the Hosoya Triangle}
	\author{Rigoberto Fl\'orez}
	\address{Department of Mathematical Sciences\\
		The Citadel\\
		Charleston, SC \\
		U.S.A.}
	\email{rigo.florez@citadel.edu}
	\thanks{  }
	\author{Robinson A. Higuita }
	\address{Instituto de Matem\'aticas\\
	Universidad de Antioquia\\
	Medell\'in\\
	Colombia\\ }
	\email{robinson.higuita@udea.edu.co}
	\thanks{     }
	\author{Antara Mukherjee}
	\address{Department of Mathematical Sciences\\
		The Citadel\\
		Charleston, SC \\
		U.S.A.}
	\email{antara.mukherjee@citadel.edu}
\end{center}

\begin{abstract}
The \emph{Hosoya triangle} is a triangular array where every entry is a product of two Fibonacci numbers. We use the geometry of this triangle to 
find new identities related to Fibonacci numbers. We give geometric interpretation for some well-known identities of Fibonacci numbers. For instance, 
 the Cassini identity and the Catalan identity.  We also extend some identities that hold in the Pascal triangle to the Hosoya triangle.  For example, the  
 hockey stick extends from binomials to products of Fibonacci numbers and the rhombus property extends a binomial identity from the Pascal triangle  
 to an identity of products of Fibonacci numbers in the Hosoya triangle.   
\end{abstract}

\maketitle
\today

\section {Introduction}

The \emph{Hosoya triangle}, denoted by $\mathcal{H}$, is a triangular array where every entry is a product of  two Fibonacci numbers (see   \cite{BlairRigoAntaraHoneycombs, Blair, BlairRigoAntara, BlairRigoAntaraGP, Ching, florezHiguitaJunesGCD, florezjunes, hosoya, koshy}).
Figure \ref{HosoyaTriangleIN} Part (a) shows that the fourth entry in row eight is the product $F_5$ and $F_4$. 

These types of triangles are ideal to provide geometric interpretation of Fibonacci number identities. For instance, in this paper we give geometric  
interpretation for the well-known identities: the Cassini identity and the Catalan identity.  For example, 
Figure \ref{HosoyaTriangleIN} Part (b) depicts some examples of Cassini identity $F_{n-1}F_{n+1}-F_{n}^2=(-1)^n$. From Part (b) we can see that 
$ F_{5}F_{3}-F_{4}^2=10-9=(-1)^4$. Similarly, we can represent geometrically the Catalan identity $F_{n-r}F_{n+r}-F_{n}^2=(-1)^{n-r-1}F_{r}^2$ 
(see Figure \ref{CasiniCatalan}).

\begin{figure}[!ht]
	\centering
	\includegraphics[scale=0.7]{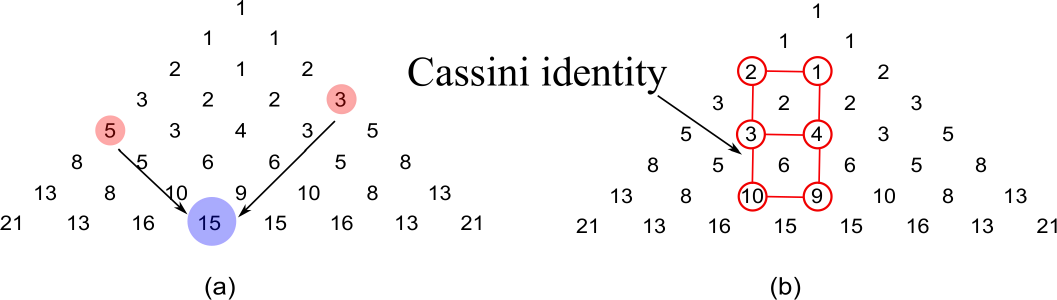}
	\caption{Hosoya triangle.} \label{HosoyaTriangleIN}
\end{figure}
In this paper we use the geometry of the Hosoya triangle to discover new Fibonacci identities and also we give geometric interpretations of some  
well-known identities. Since the proofs here  are based on the geometry of the triangle, they are simple and easy to understand. For example,  
in this paper we use the geometric representation of the product of Fibonacci numbers in the triangle to prove these identities. 
\begin{itemize}
\item  If  $a+b=c+d$ and $a\le b$, then 
$$\sum_{j=0}^{2k-1} F_{a+j}F_{b+j}=\sum_{j=0}^{2k-1} F_{c+j}F_{d+j}= F_{b+2k}F_{b+2k-1}-F_{b}F_{a-1},$$ 

\item if $m$ is a positive integer, then 
\[
\sum_{k=1}^{l} \frac{F_{m-k}+(-1)^kF_{m+k}}{F_{m}}=
\begin{cases}
-(F_{l}+F_{l-2}-1),& \mbox{ if  $l$ is odd;}\\
5F_{l-1}F_{l}+1+(-1)^{l+1}, & \mbox{if $l$ is even}
\end{cases}
\]
\end{itemize}
The first of the two identities above is a generalization of this identity  (see \cite{koshy})
$$\displaystyle\sum_{j=0}^{2k-1} F_{j+1}^2 = F_{2k+1}F_{2k}.$$ 
 
From \cite{BlairRigoAntaraHoneycombs, Blair, BlairRigoAntara, BlairRigoAntaraGP, florezHiguitaJunesGCD, florezjunes} we have observed that some  
properties that hold in the Pascal triangle also extend to other triangles including the Hosoya triangle.  In this paper we use the geometric  
representation, in the Pascal triangle (see \cite{Green}), of some binomial identities to discover new identities of Fibonacci numbers. For instance,  
the hockey stick property is one of the well-known properties that we successfully extend to the Hosoya triangle. The T-stick property in  the 
Pascal triangle gives rise to a triangular property here. In this paper we also study some other geometric
properties that the Hosoya triangle has. For example we give geometric proofs of the Cassini, Catalan, and Johnson identities.

We have found that if a rectangle is  given in a Hosoya triangle, then the differences of two of its corners points is equal to the
difference of the remaining corners points. This fundamental property allows us to have geometrical proofs of several identities.

The symmetry present in the Hosoya triangle helps us to explore several patterns, and many identities.
The rectangle property gives rise to other geometrical configurations and therefore, more identities associated with those
configurations.

\section{The Hosoya triangle and its coordinate system} \label{CoordinateSystem}

The construction presented in this section can be found in articles by Fl\'orez \emph{et al.} \cite{florezHiguitaJunesGCD} and Hosoya
\cite{hosoya}. Other similar constructions are also presented in Koshy \cite{koshy} or in \cite{BlairRigoAntaraHoneycombs, BlairRigoAntaraGP, BlairRigoAntara, Blair, Ching, florezHiguitaMukherjee, florezjunes}.
The \emph{Hosoya sequence} $\left\{H(r,k)\right\}_{r,k\ge 0}$ is defined using the double recursion
\[ H(r,k)= H(r-1,k)+H(r-2,k)  \; \text{ and } \;
 H(r,k)= H(r-1,k-1)+H(r-2,k-2),\]
with initial conditions
$H(0,0)=0; \quad H(1,0)=0; \quad H(1,1)=0; \quad H(2,1)=1,$
where $ r>1 $ and $0\le k \le r-1$. This sequence gives rise to the \emph{Hosoya triangle}, where the entry in position $k$
(taken from left to right) of the $r${th} row is equal to $H(r,k)$ (see Tables \ref{tabla1} and \ref{CasiniCatalan}, and Sloane
\cite{sloane} at \seqnum{A058071}). For simplicity in this paper we use $\mathcal{H}$ to denote the Hosoya triangle.

\begin{table} [!ht] \small
\centering
\addtolength{\tabcolsep}{-3pt} \scalebox{.85}{
\begin{tabular}{ccccccccccccc}
&&&&&&                                                               $H(0,0)$                                                 &&&&&&\\
&&&&&                                                  $H(1,0)$     &&     $H(1,1)$                                        &&&&&\\
&&&&                                         $H(2,0)$    &&     $H(2,1)$     &&     $H(2,2)$                               &&&&\\
&&&                                 $H(3,0)$   &&     $H(3,1)$     &&     $H(3,2)$      &&    $H(3,3)$                     &&&\\
&&                        $H(4,0)$     &&     $H(4,1)$    &&     $H(4,2)$     &&     $H(4,3)$     &&     $H(4,4)$             &&\\
&            $H(5,0)$     &&    $H(5,1)$    &&     $H(5,2)$     &&     $H(5,3)$      &&    $H(5,4)$     &&    $H(5,5)$    & \\
$H(6,0)$     &&    $H(6,1)$    &&     $H(6,2)$     &&     $H(6,3)$      &&    $H(6,4)$     &&    $H(6,5)$    &&    $H(6,6)$     \\
\end{tabular}}
\caption{Hosoya triangle  $\mathcal{H}$.} \label{tabla1}
\end{table}

\begin{proposition}[\cite{hosoya,koshy}]\label{lemma0} $H(r,k)= F_kF_{r-k}$.
\end{proposition}

Proposition \ref{lemma0} (represented in Figure \ref{HosoyaTriangleIN} Part (a)) gives rise to another coordinate system
(see also Fl\'orez \emph{et al.} \cite{florezHiguitaJunesGCD, florezjunes}). If $P$ is a point in  $\mathcal{H}$,
then it is clear that there are two unique positive integers $r$ and $k$ such that
$P=H(r,k)$ with $k \leq r$. From $H(r,k)=F_{k}F_{r-k}$ it is easy to
see that an $n$th \emph{diagonal} in $\mathcal{H}$ is the collection of all
Fibonacci numbers multiplied by $F_n(x)$.  For example, from Table \ref{tabla1} we can see that
the diagonal  $H(4,1)$,  $H(5,2)$, $H(6,3)$,  $H(7,4)$, $H(8,5)$,  $H(9,6), \dots$ is
equal to the diagonal $2, 2, 4, 6, 10, 16, \dots$ in Figure  \ref{CasiniCatalan}
which results from multiplying the Fibonacci sequences by $F_3=2$.

\section{Geometric properties in the Hosoya triangle}
A parallel configuration of points in the Hosoya triangle is called a \emph{ladder configuration},
for simplicity we are going to refer to this as a ladder. A \emph{rung} is the set of points on a
line intersecting both parallel (the right-up) configurations of the ladder. See Figures \ref{proofLemma2}(a),
\ref{ProofTheorem3Part1}, \ref{ProofTheorem1Part5Part6}, and \ref{ConfigurationGeneralizedFibonacci}.
The \textit{length} of the rung is the difference of its end points. The \textit{absolute length} of a rung is the absolute value of its length.

In this section we use the ladder configuration to explore geometric and algebraic properties in the
Hosoya triangle. The properties here in this paper can be easily extended to the Hosoya polynomial triangle
(see Fl\'orez \emph{et al}. \cite{florezHiguitaMukherjee}).

We first prove a lemma that will be helpful in proving several results in this paper.  
If  $L$ is a horizontal ladder in $\mathcal{H}$ where its rungs have exactly two points, then a rung sum is
a Fibonacci number and it is the same for every rung. The identity (8) in Vajda \cite{Vajda} states that 
$G_{n+m}=F_{m-1}G_n+F_mG_{n+1}$, where $G_i$ is any generalized Fibonacci sequence. This lemma gives a 
geometric interpretation of the identity in Vajda when $m=j+1$, $n=k-j$ and $G_{n}=F_{n}$.

\begin{lemma} \label{ladder:two:points:rung}
If $k,j \in \mathbb{Z}_{>0}$ and $j+i+1\le k$ , then in $\mathcal{H}$ this holds 
$$F_{j}F_{k-j}+F_{j+1}F_{k-j+1}=F_{j+i}F_{k-j-i}+F_{j+i+1}F_{k-j-i+1}=F_{k+1}.$$

Equivalently, $H(k,j)+H(k+2,j+1)=H(k,j+i)+H(k+2,j+i+1)=F_{k+1}.$

\end{lemma}

\begin{proof}
First, we take two consecutive rungs of $L$ forming a square (see Figure  \ref{proofLemma2}(a)).
Observe that each diagonal (slash and backslash) of this square has three points ---two corner points
and one inner point. Those two diagonals intersect in the inner point $p$. From the recursive definition
of the entries of $\mathcal{H}$ and the point $p$, it is easy to see that the difference of the
two corner points of the backslash diagonal of the square is equal to the difference of the corner points of
the slash diagonal of the square. This implies that sum of the points in any two consecutive rungs have the same
value. Using an inductive argument we can extend the result for any two arbitrary rungs.
Since this is true for any rung of $L$, it is true for the first rung on the left where
one of the two points is zero and the other is the Fibonacci number $F_{k+1}=H(k+2,1)$.
\end{proof}

An alternate (technical) proof can be found using the recursive definition of the Hosoya
triangle and first proving that  $H(r,k)+H(r+2,k+1)=H(r,k+1)+H(r+2,k+2)$.

All horizontal rungs in a vertical ladder in $\mathcal{H}$ have the same length except by the order of their measure (see Figures \ref{proofLemma2}(b) 
and \ref{ProofTheorem3Part1}). This result is formally stated in Proposition \ref{PropiedadDelRectangle}.

\begin{figure}[!ht]
	\centering
	\includegraphics[width=7.8cm]{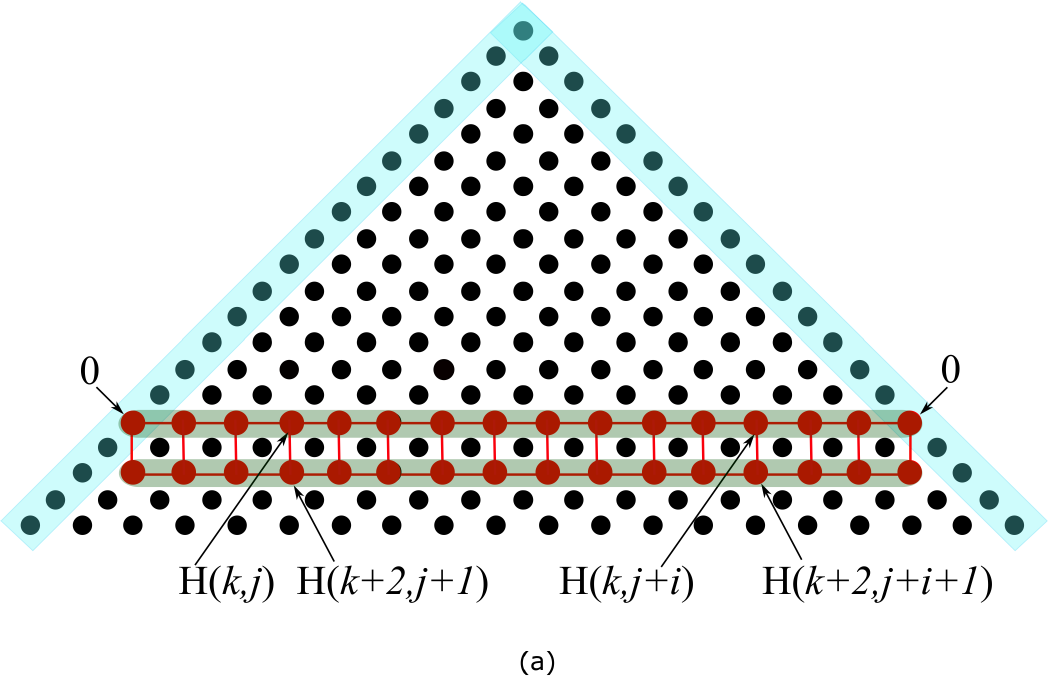}\hspace{0.25cm}	\includegraphics[width=7.8cm]{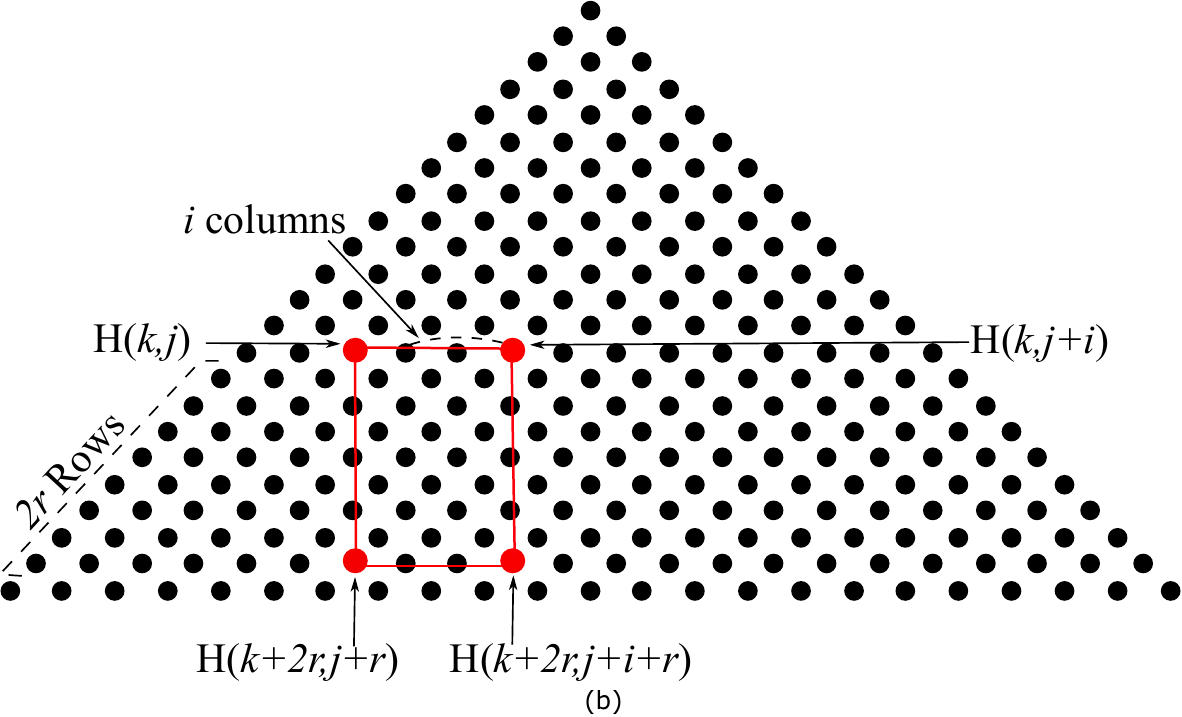}
	\caption{Rectangle property and the rung sum of a ladder.}\label{proofLemma2}
\end{figure}

\begin{proposition}[Rectangle Property]\label{PropiedadDelRectangle}
 In $\mathcal{H}$ it holds that
 $$F_{j}F_{k-j}-F_{j+i}F_{k-j-i}=(-1)^{r}(F_{j+r}F_{k+r-j}-F_{j+i+r}F_{k+r-j-i})=(-1)^{j+1}F_{i}F_{k-2j-i}.$$
 
 Equivalently,
		\[H(k,j)-H(k,j+i)=(-1)^{r}(H(k+2r,j+r)-H(k+2r,j+i+r))=(-1)^{j+1}H(k-2j,i).\]

\end{proposition}

\begin{proof}
From Figure \ref{ProofTheorem3Part1} and
Lemma \ref{ladder:two:points:rung} we can observe that
\[|a_{0}-b_{0}| =|a_{1}-b_{1}|=|a_{2}-b_{2}|= \dots=|a_{i}-b_{i}|.\]
From Figure \ref{ProofTheorem3Part1} we can see that, in particular, if we take $a_{0}=0$, then $b_{0}=H(k-2j,i)$.
\end{proof}

\begin{figure}[!ht]
	\centering
	\includegraphics[width=10cm]{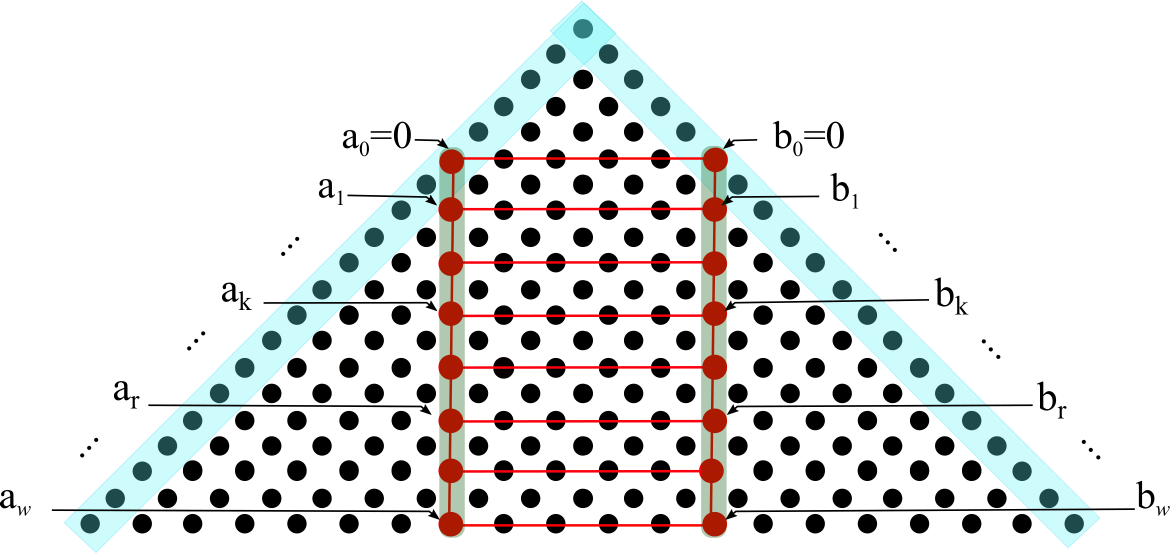}
	\caption{Every rung in any vertical ladder in $\mathcal{H}$ has the same length.}\label{ProofTheorem3Part1}
\end{figure}

The following result provides several identities in the Hosoya triangle. In particular it shows that the alternating sum of the points in a horizontal  
rung of a vertical ladder and the sum of the points in a vertical rungs of a horizontal ladder is a constant provided that the rung has even number  
of points in each case. We also see that the absolute length of each rung in a horizontal ladder is the same if there are an odd number of points  
in the rungs. Finally, if the ladders are oblique (see Figure \ref{ProofTheorem1Part5Part6}), then the absolute length of each rung is the absolute  
length of the first rung multiplied by a Fibonacci number and the sum of the points in the oblique rungs equals the sum of the points in the second  
rung multiplied by a Fibonacci number.

We may also use Proposition \ref{lemma0} to give an algebraic reinterpretation of the results  mentioned above in terms
of Fibonacci numbers.

\begin{theorem}\label{Propiedad:del:cuadrado} In the Hosoya triangle $\mathcal{H}$ these hold,
	\begin{enumerate}[(1)]
		 \item   If $r, k>0$, $j\ge 0$ and $0<2n-1\le k$ for some $n$, then
	  \[\left | \sum_{t=0}^{2n-1} (-1)^{t}H(k,j+t)\right |=\left | \sum_{t=0}^{2n-1} (-1)^{i}H(k+2r,j+t+r)\right |.\]
	
	  Equivalently, $\left | \sum_{t=0}^{2n-1} (-1)^{t}F_{j+t}F_{k-j-t}\right |=\left | \sum_{t=0}^{2n-1} (-1)^{i}F_{j+r+t}F_{k+r-j-t}\right | .$

		\item \label{Propiedad:del:cuadradoPart2}  If  $m, k>0$, $j\ge 0$ and $0<2n-1\le k$ for some $n$, then 
				\[ \sum_{t=0}^{2n-1} H(k+2t,j+t)=\sum_{t=0}^{2n-1} H(k+2t,j+m+t).\]
			
				Equivalently, $\sum_{t=0}^{2n-1} F_{j+t}F_{k+t-j}=\sum_{t=0}^{2n-1} F_{j+m+t}F_{k+t-j-m}.$

		\item  If $i$ is a positive even number, then 
		\[H(k+2i,j+i)-H(k,j)=H(k+2i,j+n+i)-H(k,j+n)=H(k+2i,i).\]
		
		Equivalently,
		$
		\det \left[\begin{array}{ll}
		F_{j+i} & F_{j} \\
		F_{k-j} & F_{k+i-j}
		\end{array}\right]
		=
		\det \left[\begin{array}{ll}
		F_{j+n+i} & F_{j+n} \\
		F_{k-j-n} & F_{k+i-j-n}
		\end{array}\right]
		=
		\det \left[\begin{array}{ll}
		F_{i} & 0\\
		0 & F_{k+i}
		\end{array}\right].
		$	
		
		\item  If $r, k$, and $j$ are positive integers with $r\ge j$, then
		\[
		H(r+k,j)-H(r,j)=F_j( H(r+k-j+1,1)-H(r-j+1,1) ).
		\]
		
		Equivalently,
		$
		\det \left[\begin{array}{ll}
		F_{j} & F_{r-j} \\
		 F_{j} &F_{r+k-j} 
		\end{array}\right]
		= F_{j}\left(F_{r+k-j}-F_{r-j} \right).
		$
            \item   If $i,j$, and $k$ are positive integers, then
		\[
		\sum_{i=0}^{m} H(k+i,j+i)=F_{k-j}\sum_{i=0}^{m} H(j+i+1,1).
		\]		
  \end{enumerate}	
\end{theorem}

\begin{proof}
We prove Part (1) for two consecutive rungs. The general case follows easily
using an inductive argument, so we omit it. From Figure \ref{ProofTheorem3:Part2}
and Lemma \ref{ladder:two:points:rung} we have
\[c_{1}-c_{0}=d_{0}-d_{1}; \quad c_{3}-c_{2}= d_{2}-d_{3}; \quad \dots \quad c_{2i-1}-c_{2i-2}= d_{2i-2}-d_{2i-1}.\] 
 This implies that
$| \sum_{i=0}^{2n-1} (-1)^{i} c_{i}|=|\sum_{i=0}^{2n-1} (-1)^{i} d_{i} |$.

\begin{figure}[!ht]
	\centering
	\includegraphics[width=11cm]{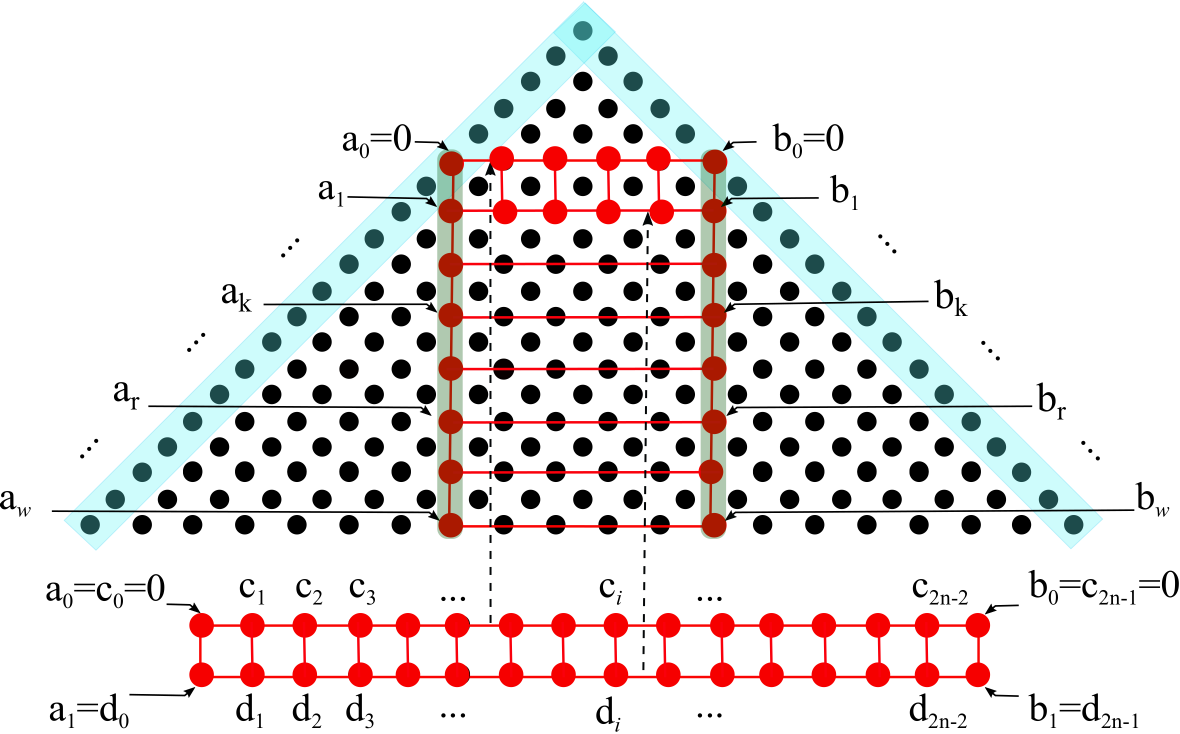}
	\caption{The alternating sum of the points in a rung is the same.}\label{ProofTheorem3:Part2}
\end{figure}

Proof of Part (2). From Figure \ref{Proof:Theorem3:Part3} and Lemma \ref{ladder:two:points:rung}
we have
\[c_{0}+c_{1}=d_{0}+d_{1};\quad c_{2}+c_{3}=d_{2}+d_{3};  \quad \dots \quad c_{2n-2}+c_{2n-1}=d_{2n-2}+d_{2n-1}.\]
This leads to the conclusion.

\begin{figure}[!ht]
	\centering
	\includegraphics[width=14cm]{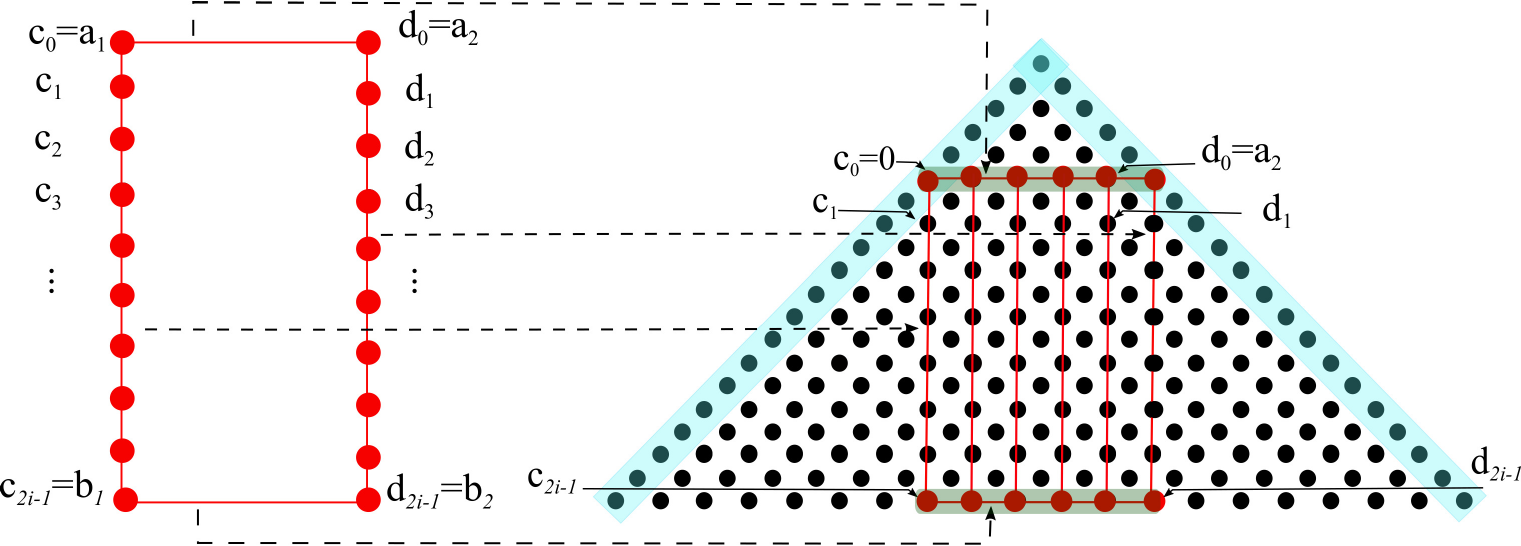}
	\caption{Rungs with even number of points have the same sum.}\label{Proof:Theorem3:Part3}
\end{figure}

Proof of Part (3). Let $a_0$ and $a_1$ be the points of the first rung of the ladder and let $b_0$ and $b_1$ be the points of the last rung of the ladder, as in Figure \ref{Proof:Theorem3Part4}. We assume that the ladder has an odd number of rungs. 
We define $x$ as the sum of all points between $c_{0}$ and $d_{0}$ (left-hand side points), including both of them, and let $y$ 
be the sum of all points between $c_{1}$ and $d_{1}$ (right-hand side points), including both of them. From Part \eqref{Propiedad:del:cuadradoPart2} 
$a_{0}+x =a_{1}+y \; \text{ and } \; b_{0}+x = b_{1}+y.$
From this it is easy to see that $a_{1}-a_{0}=x-y=b_{1}-b_{0}$. If, in particular, we take $a_{0}=H(k,0)$, we have $b_{0}=H(k+2i,i)$.

\begin{figure}[!ht]
	\centering
	\includegraphics[width=14cm]{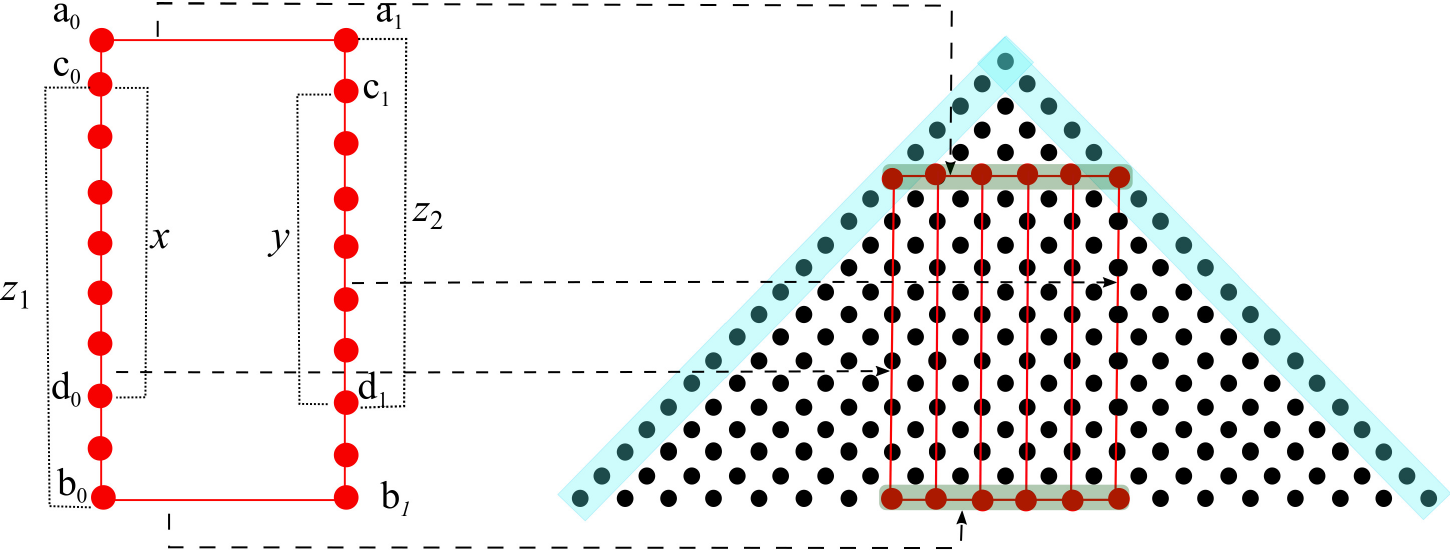}
	\caption{Rungs with odd number of points have the same length.}\label{Proof:Theorem3Part4}
\end{figure}

Proof of parts (4) and (5). Using the coordinate system described in Section \ref{CoordinateSystem} and
Figure \ref{ProofTheorem1Part5Part6} we can see that the points in the first parallel side (the right-up) of the ladder are of the form $F_{k}F_{i}$
where $F_{k}$ is fixed and the points in the second parallel side are of the form $F_{k+j}F_{i}$ where $F_{k+j}$ is fixed.
Therefore, the points in the $r$-th rungs are $F_{k}F_{r}$, $F_{k+1}F_{r}$, $F_{k+2}F_{r}, \dots$,  $F_{k+j}F_{r}$.
To prove Part (4) we first note that the length of any rung is given by $F_{k+j}F_{r}-F_{k}F_{r}=F_{r}(F_{k+j}-F_{k})$
where $F_{k+j}-F_{k}$ is the length of the first rung. The proof of Part (5) follows by adding the points of a rung. Thus,
\[F_{k}F_{r}+F_{k+1}F_{r}+F_{k+2}F_{r}+ \dots +F_{k+j}F_{r}=F_{r}\left(F_{k}+F_{k+1}+F_{k+2}+ \dots + F_{k+j}\right).\]
Note that $\left(F_{k}+F_{k+1}+F_{k+2}+ \dots +  F_{k+j}\right)$ is the sum of points in the second rung.
\end{proof}

\begin{figure}[!ht]
	\centering
	\includegraphics[width=9cm]{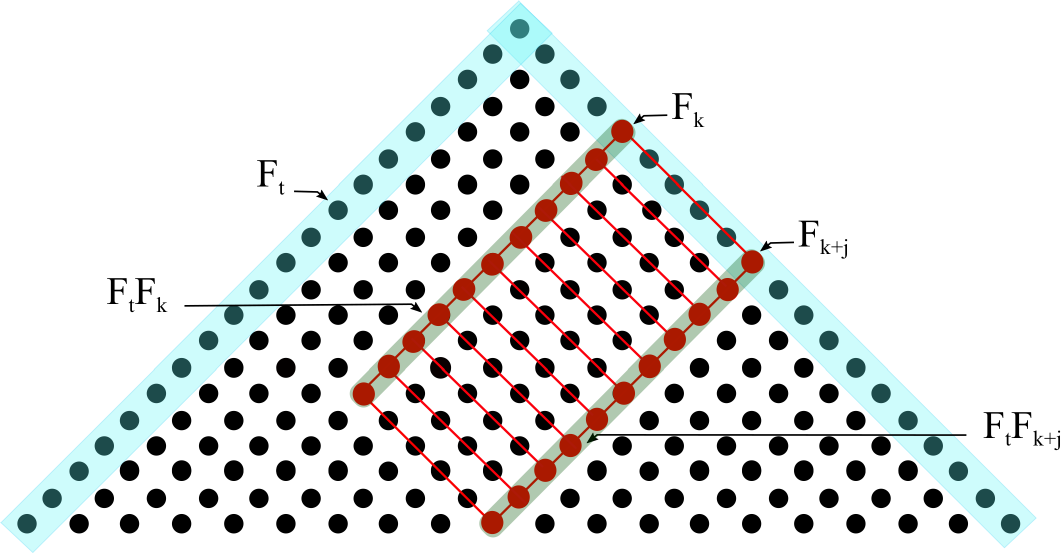}
	\caption{The sum of the points of the rungs are proportionally related.}\label{ProofTheorem1Part5Part6}
\end{figure}

In the next part we give a geometric interpretation in the Hosoya triangle of the Cassini, Catalan, and Johnson identities (see \cite{JohnsonC}).

\begin{figure}[!ht]
	\centering
	\includegraphics[scale=0.55]{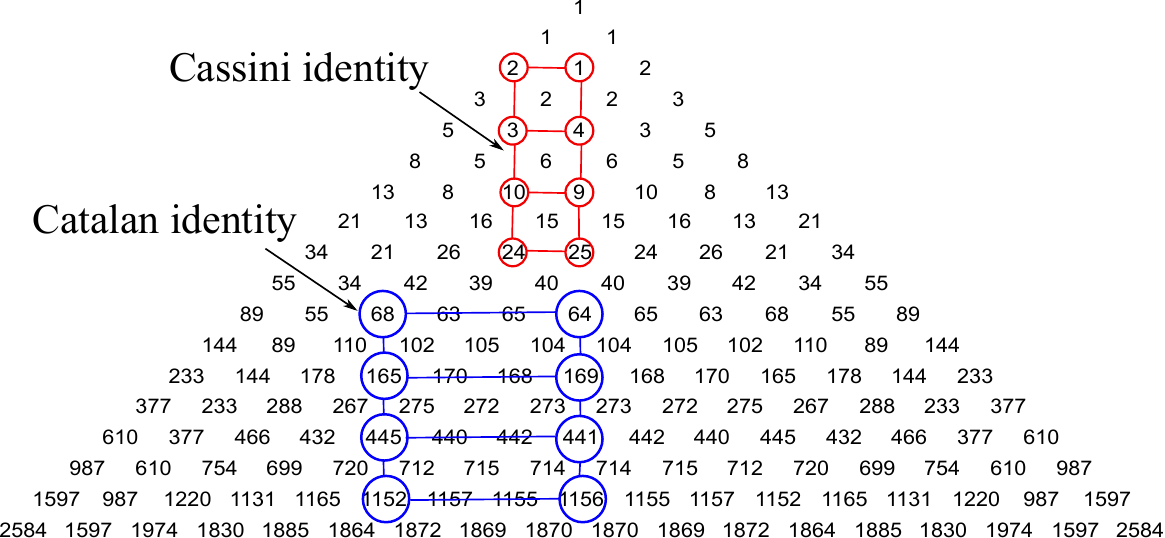}
	\caption{Geometry of the Cassini and Catalan identities.} \label{CasiniCatalan}
\end{figure}

The length of a rung in a vertical ladder in $\mathcal{H}$ gives rise to the Cassini identity, if one of the uprights
is located in a central vertical line of $\mathcal{H}$ and the rung has exactly two points (see Figure \ref{CasiniCatalan}). Thus,
$$
H(2k,k)-H(2k,k-1)=(-1)^{k-1}.
$$
The same type of ladder as above also gives rise to the Catalan identity. Thus,
$$
H(2k,k)-H(2k,k-j)=(-1)^{k-j}H(2j,j).
$$

The length of a rung in a vertical ladder in $\mathcal{H}$ gives rise to the d'Ocagne identity if one of them is
located on the top of the ladder, and the coordinates of the end points of the second rung are $H(k+j+1,j)$ and  $H(k+j+1,k)$. Thus,
$$
H(k+j+1,k)-H(k+j+1,j)=(-1)^j H(k-j+1,k-j).
$$
The length of a rung in a vertical ladder in $\mathcal{H}$ gives rise to the Johnson identity. Thus, if
$k+j=r+i$ and $i<j$, then  for every $l \le i$ it holds that
\begin{eqnarray*}
	H(k+j,j)-H(r+i,i)&=&(-1)^{l}(H(k+j-2l,j-l)-H(r+i-2l,i-l))\\
	&=&(-1)^{i}H(k+j-2i,j-i).
\end{eqnarray*}

As a corollary of Theorem \ref{Propiedad:del:cuadrado} Part (3)
we have that if the points $a_{i}$ and $b_{i}$ in the Hosoya triangle are as in Figure \ref{ConfigurationsZigzag}(e), then
$\left(a_{j}+a_{j+1}\right) -\left(b_{j}+b_{j+1}\right)$ is a constant. This property is analogous to a property in Pascal's
triangle that yields the Catalan numbers (see \cite{VanBilliard}).

\begin{proposition}\label{Prop:zigzag}
		Let $a, b, j$ be positive integers with $j\le \min\{a,b \}$. If $A(F_{a-j},F_{a+j})$, $B(F_{b-j},F_{b+j})$ and $C(F_{a},F_{b})$ are points in the Cartesian plane, then
		\begin{enumerate}
		\item the line passing through  $A$ and $B$ is parallel to the line passing through $(0,0)$ and $C$. Thus,
		
		$$ \dfrac{F_a}{F_b}=\dfrac{F_{a-j} +(-1)^j F_{a+j}}{F_{b-j} +(-1)^j F_{b+j}}.$$

		\item The triangle with base $F_{a+j}+(-1)^jF_{a-j}$ and height  $F_{b}$ has same area as the triangle with base $F_{b+j}+(-1)^jF_{b-j}$ and height  $F_a$.
		\end{enumerate}
\end{proposition}
The proof of Proposition \ref{Prop:zigzag} is easy using Proposition \ref{PropiedadDelRectangle}, therefore we omit it.

The configuration depicted in Figure \ref{ConfigurationsZigzag} Part (a) is called a \textit{zigzag}.
The configuration depicted in Figure \ref{ConfigurationsZigzag} Part (b) is called a \textit{left zigzag}.
The configuration depicted in Figure \ref{ConfigurationsZigzag} Part (c) is called a \textit{right zigzag}.
The configuration depicted in Figure \ref{ConfigurationsZigzag} Part (d) is called a \textit{long zigzag}.
There should be a finite number of points in any zigzag configuration.

The configuration depicted in Figure \ref{ConfigurationsBrad} Part (a) is called a \textit{braid (or hourglass)}. 
The configuration depicted in Figure \ref{ConfigurationsBrad} Part (b) is called a \textit{left braid}.
The configuration depicted in Figure \ref{ConfigurationsBrad} Part (c) is called a \textit{right braid}.
There should be a finite number of points in any braid configuration.

\begin{figure}[!ht]
\centering
		\includegraphics[width=10cm]{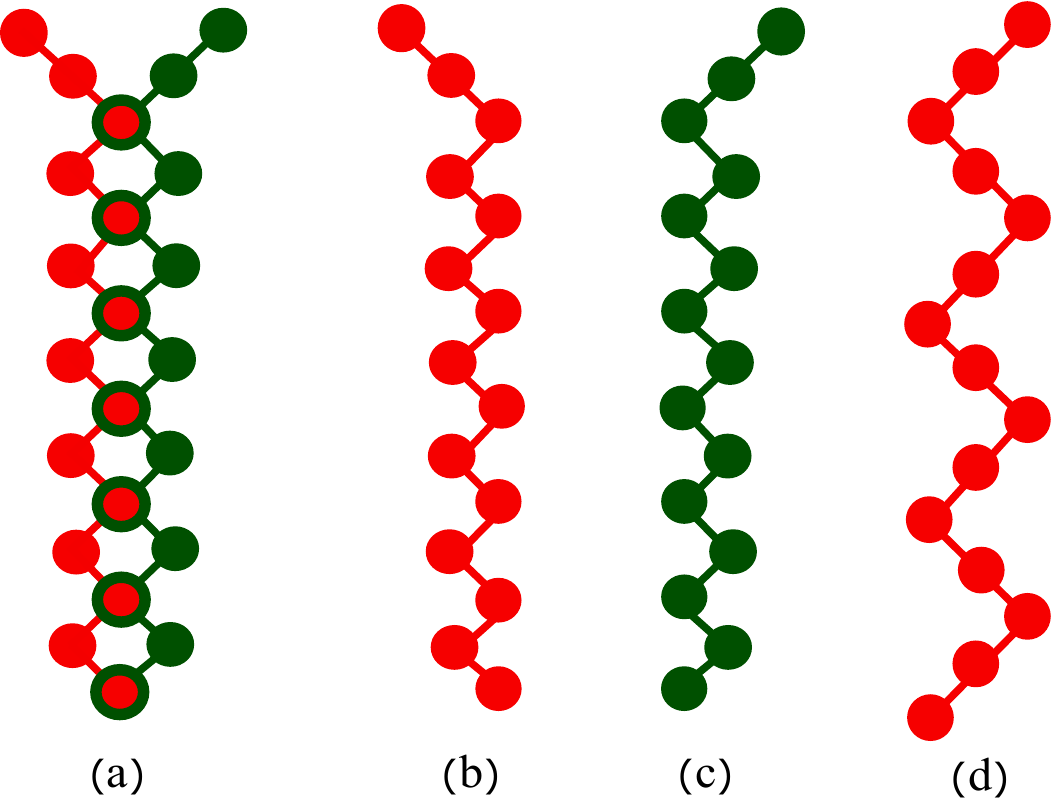}
	\caption{Zigzag configurations.} \label{ConfigurationsZigzag}
\end{figure}

\begin{corollary} \label{cor:6}The sum of alternating points of a  long zigzag configuration in $\mathcal{H}$ starting
	from its second point is equal to the difference of the last point and the first point of the zigzag
	configuration.
	Moreover, any column of points forming a rectangle with the column of alternating
	points has the same sum (see Figure \ref{ZigZagSuma}(a) and Figure \ref{ConfigurationsZigzag} Part (d)).
	More precisely, if $a, b, c$, and $d$ are positive integers such that $a+b=c+d$ and $a\le b$, then for every
	positive integer $k$ it holds that
	\[
	\sum_{j=0}^{2k-1} F_{a+j}F_{b+j}=\sum_{j=0}^{2k-1} F_{c+j}F_{d+j}= F_{b+2k}F_{b+2k-1}-F_{b}F_{a-1}.
	\]
\end{corollary}

\begin{proof}
  We prove that the sum of alternating points of a long zigzag configuration in $\mathcal{H}$
  starting from its second point is equal to the difference of the last point
with the first point of the long zigzag. The last part of this corollary follows from Theorem \ref{Propiedad:del:cuadrado} Part (2).

Suppose
$p_{1}, p_{2}, p_{3}, \dots, p_{n-2}, p_{n-1}, p_{n}$ are the points of the long zigzag ordered from
top to bottom where $p_{1}$ is the first point  on the top and $p_{n}$ is the last
point in the bottom (see Figure \ref{ZigZagSuma}(a)). We  want to show that $p_{2}+ p_{4}+ p_{6}+ \dots+  p_{n-1}= p_{n}-p_{1}$.
From definition of $\mathcal{H}$ on Page~\pageref{CoordinateSystem} we know that $H(r,k)= H(r-1,k)+ H(r-2,k)    \text{ and } H(r,k)= H(r-1,k-1)+H(r-2,k-2)$.
This implies that
\begin{eqnarray}
p_{3}&=&p_{1}+p_{2}, \label{Ecuacion1}\\
p_{5}&=&p_{3}+p_{4}, \label{Ecuacion3}\\
p_{7}&=&p_{5}+p_{6}, \label{Ecuacion5}\\
\vdots && \vdots  \nonumber\\
p_{n}&=&p_{n-2}+p_{n-1}. \label{Ecuacionn}
\end{eqnarray}
Substituting Equation  \eqref{Ecuacion1} into $p_{2}+ p_{4}+ p_{6}+ \dots+  p_{n-1}$ we obtain
\[p_{2}+ p_{4}+ p_{6}+ \dots+  p_{n-1}= -p_{1}+ p_{3}+ p_{4}+ p_{6}+ \dots+  p_{n-1}.\]
Substituting Equation  \eqref{Ecuacion3} into the right-hand side of this equality, we obtain
\[p_{2}+ p_{4}+ p_{6}+ \dots+  p_{n-1}= -p_{1}+ p_{5}+ p_{6}+ \dots+  p_{n-1}.\]
Substituting Equation  \eqref{Ecuacion5} into the right-hand side of this equality, we obtain
\[p_{2}+ p_{4}+ p_{6}+ \dots+  p_{n-1}= -p_{1}+ p_{7}+ p_{8}+ \dots+  p_{n-1}.\]
 We systematically keep making these substitutions to obtain
\[p_{2}+ p_{4}+ p_{6}+ \dots+  p_{n-1}= -p_{1}+  p_{n}.\]
This completes the proof of the corollary.
\end{proof}

The identity in the previous corollary is a generalization of the identities found in Koshy \cite{koshy} and Vajda \cite{Vajda}. 
For example, if $a=b=1$, then we obtain 
$$\displaystyle\sum_{j=0}^{2k-1} F_{1+j}^2=F_{2k}F_{2k+1}.$$ 
In additon, if  $a=1, b=2$, then 
$$\displaystyle\sum_{j=0}^{2k-1}F_{j+1}F_{j+2}=F_{2k+2}F_{2k+1}$$ and so on.

\begin{figure}[!ht]
	\centering
	\includegraphics[width=8cm]{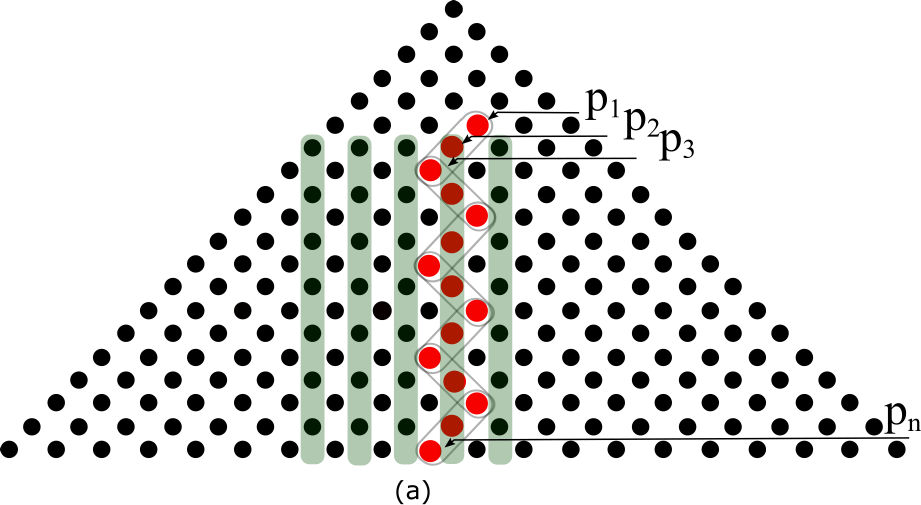}\hspace{0.2cm}	\includegraphics[width=8cm]{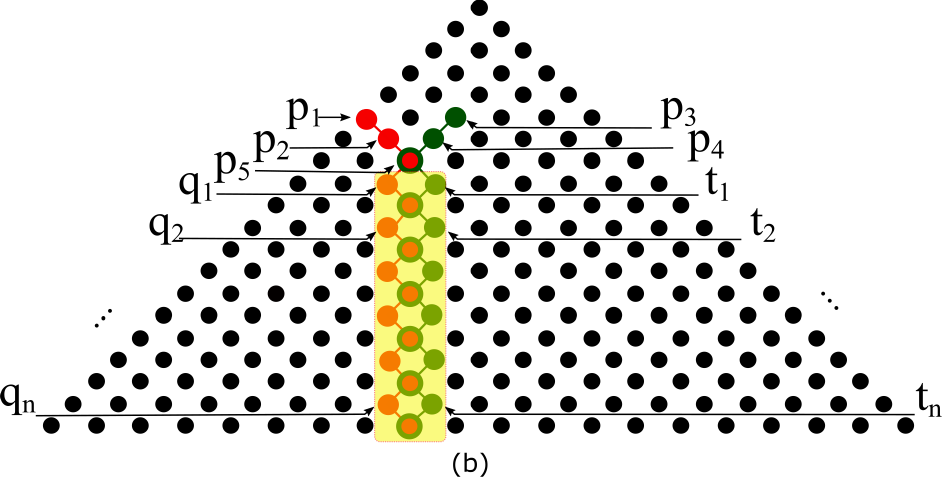}
	\caption{Zigzag Configurations.}
	\label{ZigZagSuma}
\end{figure}

\begin{theorem}[Zigzag property] \label{ZigZagTheorem} If a zigzag configuration with $6k+5$ points holds in $\mathcal{H}$, then the sum of all points in its left zigzag is equal
	to the sum of all points in its right zigzag (see Figure  \ref{ConfigurationsZigzag}  and Figure \ref{ZigZagSuma}(b)).
	\end{theorem}
		
	\begin{proof} From Figure \ref{ZigZagSuma}(b) and definition of $\mathcal{H}$ on Page~\pageref{CoordinateSystem}
	it is easy to see that $p_{1}+p_{2}=p_{3}+p_{4}=p_{5}$. Since the zigzag configuration has $6k+5$ points,
	there remain $6k$ points distributed in three vertical equal sets (see Figure \ref{ZigZagSuma}(b) for the labeling of points).
	So, every set has an even number of points. This and Theorem \ref{Propiedad:del:cuadrado} Part (2) imply that
	$q_{1}+q_{1}+\dots+ q_{n} = t_{1}+t_{1}+ \dots+ t_{n}$. This completes the proof.      	
	\end{proof}

As a corollary of Theorem \ref{ZigZagTheorem} we can see that the Hockey Stick property seen in Figure \ref{HockeyStickFig}(a) and originally found in the Pascal triangle (see \cite{BlairRigoAntaraGP,VanBilliard}), also holds in $\mathcal{H}$.
 Thus, the sum of all points in the shaft of a hockey stick is equal
to the point on the blade of the hockey stick. The blade of the hockey stick is going to the left or to the right depending on the numbers
of points that are on the shaft. If we consider the hockey stick configuration on one side of the Hosoya triangle ($\mathcal{H}$),
then (by the symmetry of $\mathcal{H}$) when the same configuration is represented on the other side of $\mathcal{H}$, the blade of
the hockey stick changes direction (from left to right or from right to left).  We use $s_{1}, s_{2}, \cdots, s_{i}$ to represent the points on the shaft of the hockey stick. We use $b_{L}$ and $b_{R}$ to
represent the point on the blade of the hockey stick. We use $b_{L}$ to indicate that it is on the left-hand side of the hockey stick and $b_{R}$
indicates that it is on right side of the hockey stick. (See Figure \ref{HockeyStickFig}(a).) These give that if $s_{1}, s_{2}, \cdots, s_{i}$ 
are the points on the shaft of the hockey stick where $s_{1}$ is a point on one of the edges of $\mathcal{H}$ and $b_{t}$ is the point on the
blade of the hockey stick, with $t \in \{L, R\}$ (see Figure \ref{HockeyStickFig}(a)), then 

\begin{enumerate}
  \item If the hockey stick is on the left-hand side of $\mathcal{H}$, then $s_{1}+s_{2}+ \cdots+s_{2n}=b_{L}$, or if the hockey stick is on the right-hand side of $\mathcal{H}$, then $s_{1}+s_{2}+ \cdots+s_{2n}=b_{R}$. Therefore, we obtain this identity      
 \[\sum _{i=0}^{k-1} F_{i+1} F_{i+r-1}=F_k F_{k+r-1}.\]
  
  \item If the hockey stick is on the left-hand side of $\mathcal{H}$, then  $s_{1}+s_{2}+ \cdots+s_{2n+1}=b_{R}$, or if the hockey stick is on the 
  right-hand side of $\mathcal{H}$,
  then $s_{1}+s_{2}+ \cdots+s_{2n+1}=b_{L}$. Therefore, we obtain this identity
 \[\sum _{i=0}^{k-1} F_{i+1} F_{i+r-1}=F_{k+1} F_{k+r-2}.\]
 
  \item If the hockey stick is in the center of $\mathcal{H}$, then $s_{1}+s_{2}+ \cdots+s_{i}=b_{L}=b_{R}$. Therefore, for every $i$ this holds 
   \[ \sum _{i=0}^{k-1} F_{i+1}^2= F_{k}F_{k+1}.\]
  
\end{enumerate}
The first two of the three identities above are generalizations of identities like 
$\sum _{i=0}^{k-1} F_{i+1} F_{i}=F_k^2 $ 
if we have $r=1$, $\sum _{i=0}^{k-1} F_{i+1} ^2=F_{k+1} F_{k}$ if we have $r=2$) and so on.
 The third identity (above), is well-known and can be found in \cite{koshy}.
 
We now formally state the previous results in the following corollary. 

\begin{corollary}[Hockey Stick property] \label{HockeyStickCor} Let $r,k\in \mathbb{Z}_{> 0}$, where $k$ is the number of entries in the shaft of the hockey stick. Then these hold in $\mathcal{H}$.

\begin{enumerate}
\item If $k$ is even and the hockey stick is on  left-hand side of the median of $\mathcal{H}$, then 
\[\sum _{i=0}^{k-1}H(r+2i,i+1)=H(r+2k-1,k).\]
\item If $k$ is even and the hockey stick is on right-hand side of the median of $\mathcal{H}$, then
\[\sum _{i=0}^{k-1}H(r+2i,r+i-1)=H(r+2k-1,r+k-1).\]
\item If $k$ is odd and the hockey stick  is on  left-hand side of the median of $\mathcal{H}$, then
\[\sum _{i=0}^{k-1}H(r+2i,i+1)=H(r+2k-1,k+1).\]

\item If $k$ is odd and the hockey stick is on right-hand side of the median of $\mathcal{H}$, then
\[\sum _{i=0}^{k-1}H(r+2i,r+i-1)=H(r+2k-1,r+k-2).\]

\item If the hockey stick is in the median, then for $k>0$ this holds
\[\sum _{i=0}^{k-1}H(2i+2,i+1)=H(2k+1,k)=H(2k+1,k+1).\]

\end{enumerate}

\end{corollary}
\begin{figure}[!ht]
\centering
		\includegraphics[width=8cm]{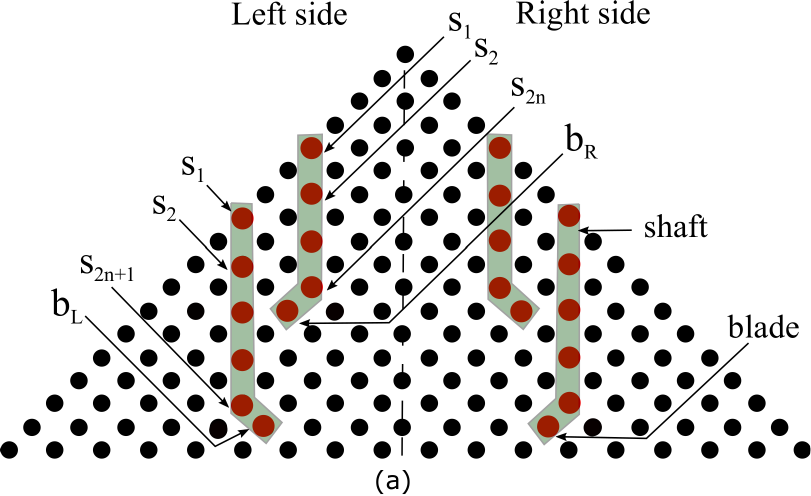}\hspace{0.25cm}\includegraphics[width=8.2cm]{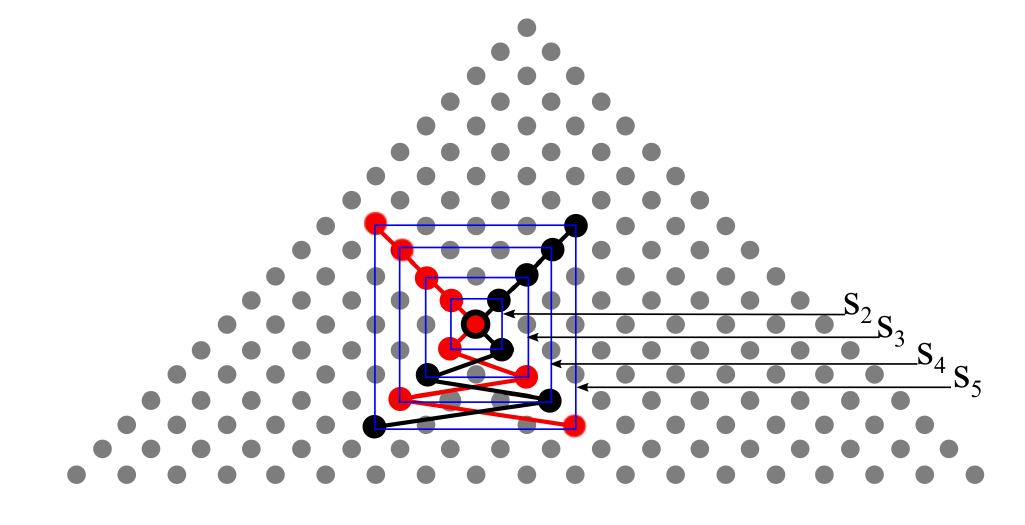}
	\caption{Hockey Stick and Braid configurations.} \label{HockeyStickFig}
\end{figure}

\begin{proposition}[Braid property] If the braid configuration holds in $\mathcal{H}$,
then the sum of all points in the left braid is equal to the sum of all points in
the right braid (see Figure \ref{ConfigurationsBrad} and Figure \ref{HockeyStickFig}(b))
\end{proposition}
	
\begin{proof} In Figure \ref{HockeyStickFig}(b) we observe that the left braid configuration is formed
by all left corner points of the squares of even side length  and all corner points in backslash
diagonal in the squares of odd length
(see for example $S_{2}, S_{3},S_{4}$, and $S_{5}$ in Figure \ref{HockeyStickFig}(b)).
The right braid configuration is formed, similarly, by all right corner points of the squares of even side length
and all points in the slash diagonal in the  squares of odd side length.

From Proposition \ref{PropiedadDelRectangle} it is easy to deduce that in a square configuration
in $\mathcal{H}$ with even side length, it holds that the sum of two vertical corner points is equal
to the sum of the remaining corner points. If the square configuration has odd side length,
then it holds that the sum of two corner points in a diagonal of the square is equal to the
sum of the remaining corner points in the other diagonal. Using this property and
Figure \ref{HockeyStickFig}(b) we can see that the sum of left corner points of the innermost
square $S_{2}$ is equal to the sum of its right corner points. We now observe that the
square $S_{3}$ has odd side length. Therefore, the sum of the corner points in the slash
diagonal equals the corner points in the backslash diagonal. The square $S_{4}$ satisfies
the property that the sum of the vertical left corner points equals the sum of the right corner points.
The square $S_{5}$ satisfies the property that the sum of the corner points in the slash diagonal
equals the sum of the remaining corner points in the backslash diagonal. We can continue this process
inductively as long as it is required by the braid configuration embedded in $\mathcal{H}$.
From this, Figure \ref{HockeyStickFig}(b), and the observation given in the first paragraph
it is easy to obtain the conclusion of the proposition.
\end{proof}	

\begin{figure}[!ht]
\centering
		\includegraphics[width=10cm]{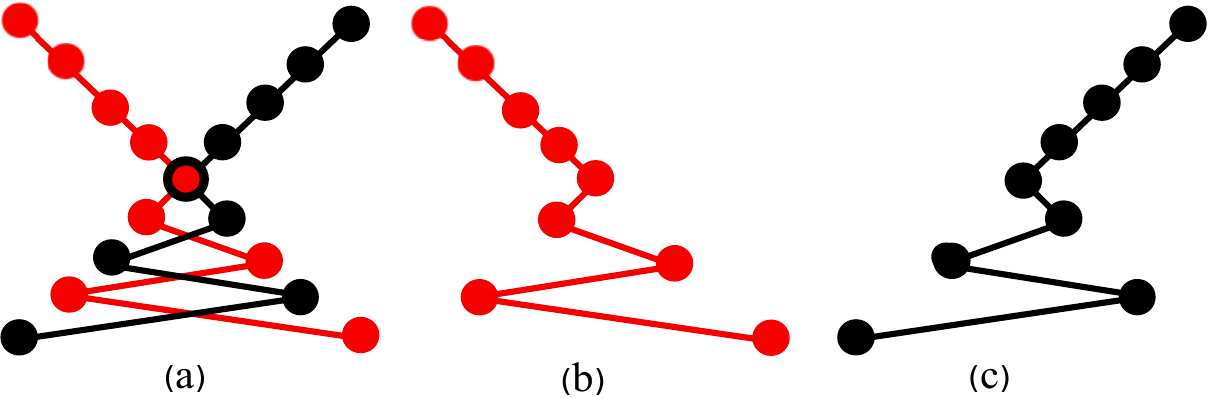}
	\caption{Braid configurations.} \label{ConfigurationsBrad}
\end{figure}

For any square $S$ in Figure \ref{HockeyStickFig}(b) it holds that the sum of the corner points in one diagonal of $S$ is the additive inverse
of the sum of the corner points in  the remaining diagonal of $S$. Therefore, using all squares in Figure \ref{HockeyStickFig}(b) it holds that
$$
\sum_{k=0}^{l} H(n-k,m-k)+\sum_{k=1}^{l} (-1)^k H(n+k,m+k)
=\sum_{k=0}^{l} H(n-k,m)+\sum_{k=1}^{l} (-1)^k H(n+k,m).
$$
If in Figure \ref{HockeyStickFig}(b) we eliminate the common point  ---the point that is  the intersection of left braid and right braid--- we obtain
$$
F_{n-m}\sum_{k=1}^{l} (F_{m-k}+(-1)^kF_{m+k})=F_{m}\sum_{k=1}^l(F_{n-m-k}+(-1)^k F_{n-m+k} ).
$$
This with $r=n-m$ implies
$$
\sum_{k=1}^{l} \frac{F_{m-k}+(-1)^kF_{m+k}}{F_{m}}=\sum_{k=1}^l\frac{F_{r-k}+(-1)^k F_{r+k}}{F_{r}}.
$$

Therefore we have the following corollary.

\begin{corollary}
If $m,k$ are positive integers then,
$$
\sum_{k=1}^{l} \frac{F_{m-k}+(-1)^kF_{m+k}}{F_{m}}=
\begin{cases}
-(F_{l}+F_{l-2}-1),& \mbox{ if  $l$ is odd;}\\
5F_{l-1}F_{l}+1+(-1)^{l+1}, & \mbox{if $l$ is even}
\end{cases}
$$

\end{corollary}

This is a previously unknown identity.

\section{Properties of the Pascal triangle that extend to the Hosoya triangle}
In this section we extend a few properties from the Pascal triangle to the Hosoya triangle.
These properties of the Pascal triangle may be found in \cite{Green}.

If we construct an oblique (backslash) ladder with horizontal rungs of length two
(see Figure \ref{ConfigurationGeneralizedFibonacci}), then the ladder gives rise to
generalized Fibonacci numbers. Thus, adding the two points of each rung of the ladder
gives rise to a sequence of second order which is a generalized Fibonacci sequence.
Recall that the generalized Fibonacci sequence is given by
$G_n=G_{n-1}+G_{n-2}$ with $G_1=a$ and $G_2=b$. Here we note that $a$ and $b$ are the points
in the first rung of the oblique ladder, (see Figure \ref{ConfigurationGeneralizedFibonacci}).
In particular, $a$ and $b$ are consecutive Fibonacci numbers with $a>b$. Some of the ladders
give rise to certain sequences found in Sloane \cite{sloane}. In particular, with $a=1$ and $b=1$
we obtain the Fibonacci number sequence  \seqnum{A000045}. If $a=2$ and $b=1$, we obtain the
Lucas number sequence \seqnum{A000032}. If $a=3$, $b=2$, we obtain sequence \seqnum{A013655}
which is a sequence where each term is obtained by adding a Fibonacci and a Lucas number.
The sequence \seqnum{A206610} is obtained with $a=13$ and $b=8$.

\begin{figure}[!ht]
\centering
		\includegraphics[width=160mm]{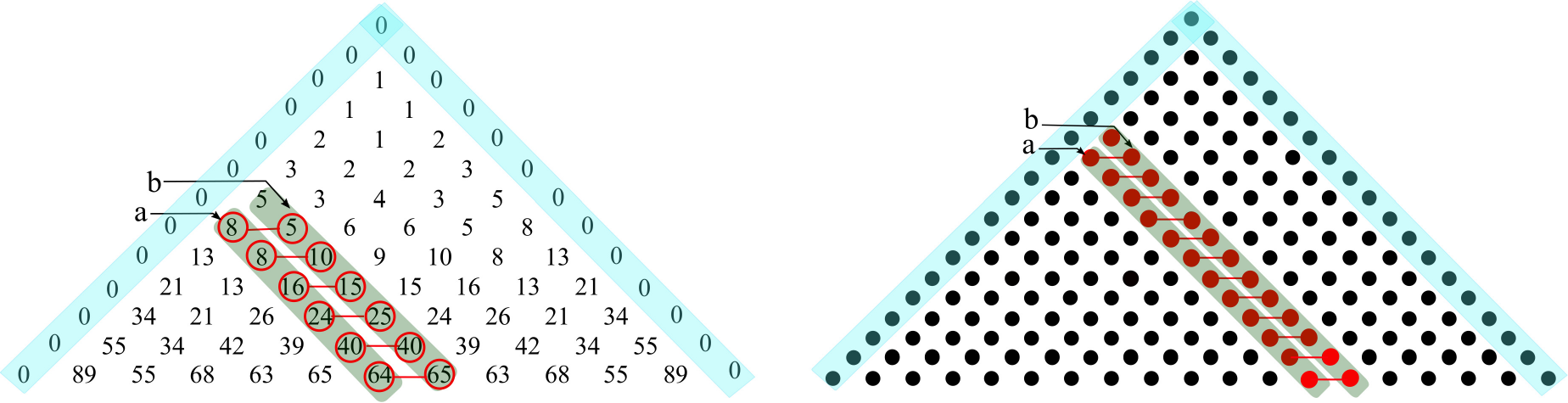}
	\caption{Generalized Fibonacci.} \label{ConfigurationGeneralizedFibonacci}
\end{figure}

If we consider two consecutive rungs of an oblique ladder ---that has horizontal rungs--- we obtain a \emph{rhombus property} (see Figure \ref{Tsticks}(b)).  
This property is an extension of a similar property found in the Pascal triangle (see \cite{Green}). Thus, the rhombus property states  
that the differences of a cross multiplication is always a point of the triangle. That is, the cross multiplication
is the determinant of the numbers present in the rungs of the oblique ladder (see Figure \ref{Tsticks}(b)). Formally we have this proposition. 

\begin{proposition}\label{RhombusPro} 
\begin{enumerate} 
\item Slash ladder. For a fixed positive integer $r$ and for any $n\ge 1$ it holds that

\[
\begin{vmatrix}
H(n,r) & H(n,r+1) \\
H(n+1,r) & H(n+1,r+1) \end{vmatrix}
= (-1)^{n-r+1}F_rF_{r+1}.
\]
\item Backslash ladder. For fixed positive integers $r$ and $n$ and for any $i\ge1$ it holds that
\[
\begin{vmatrix}
 H(n+i,r+i) & H(n+i,r+i+1) \\
 H(n+i+1,r+i+1) & H(n+r+1,r+i+2) \\
\end{vmatrix}
=
(-1)^{r+i+1} F_{n-r-1} F_{n-r}
\] 
\end{enumerate} 
\end{proposition}

\begin{proof} We prove Part (1), the proof of part (2) is similar and we omit it. We observe that the points of two consecutive rungs of a slash ladder are given by $H(n,r)$, $H(n,r+1)$ and $H(n+1,r)$, $H(n+1,r+1)$. These points give rise to a rhombus. Subtracting the product of the diagonal points of the rhombus gives $H(n,r)H(n+1,r+1)- H(n,r+1)H(n+1,r)$. This, the definition of $H(n,r)$, and the Cassini identity implies that

\begin{align*}
H(n,r)H(n+1,r+1)- H(n,r+1)H(n+1,r)&=F_rF_{n-r}F_{r+1}F_{n-r}-F_{r+1}F_{n-r-1}F_{r}F_{n+1-r}\\
						    &=F_rF_{r+1}(F_{n-r}^2-F_{n-r-1}F_{n-r+1})\\ 
						    &=F_rF_{r+1}(-1)^{n-r+1}.
\end{align*}
This completes the proof. 
\end{proof}

An additional configuration that yields a geometry and an identity that we can explore is  a \emph{triangle configuration}
 seen in Figure \ref{Tsticks}(a).

If we take the triangle configuration as in Figure \ref{Tsticks} Part (a) then $a+b-c$ is a Fibonacci number.
Note that $a$ and $b$ are points constituting the top oblique side of the triangle and the points $b$ and $c$
are points along the same vertical line at a distance two from each other. So, if
$a=H(n+1,r-1)$, $b=H(n,r)$, and $c=H(n+2,r+1)$, then $a+b-c= (-1)^{r-n+1} F_{2 r-n}$. This can be stated as the following proposition.

\begin{proposition}\label{TringlePro} $a+b-c=(-1)^rH(n-2r+2,2)$
\end{proposition}

\begin{proof} Since $a=H(n+1,r-1)$, $b=H(n,r)$, and $c=H(n+2,r+1)$, we have 
\begin{align*}
a+b-c&=a+b-(H(n+3,r+1)-H(n+1,r+1))\\
	&=b+H(n+1,r+1)+a-H(n+3,r+1)\\
	&=H(n+2,r+2)-H(n+2,r).						    
\end{align*}
This and Proposition \ref{PropiedadDelRectangle}, imply that $a+b-c=(-1)^rH(n-2r+2,2)$.
\end{proof}

\begin{figure}[!ht]
	\centering
	\includegraphics[width=16cm]{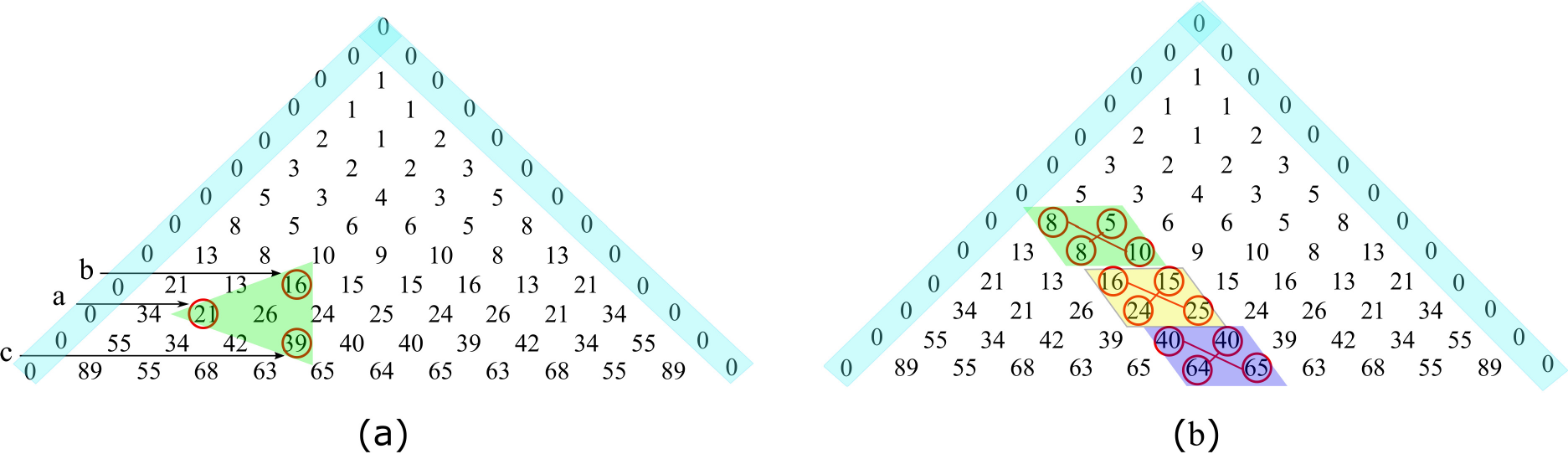}
	\caption{Triangle and rhombus  properties.} \label{Tsticks}
\end{figure}

\section{Acknowledgement}
	
The first and last authors were partially supported by The Citadel Foundation.

\bigskip
\hrule
\bigskip
	
\noindent 2010 {\it Mathematics Subject Classification}:
Primary 11B39; Secondary 11B83.
	
\noindent \emph{Keywords: }
Fibonacci numbers, Hosoya triangle,  rectangle property, zigzag property, braid property, Cassini and Catalan identities.
	

\begin{thebibliography}{99}

\bibitem{BlairRigoAntaraHoneycombs} M.~Blair, R. Fl\'orez, and A. Mukherjee, Honeycombs in the Hosoya Triangle,  
\emph{Math Horizons}, \textbf{29.3} (2022), 24--27.

\bibitem{BlairRigoAntaraGP} M.~Blair, R. Fl\'orez, and A. Mukherjee, Geometric Patterns in the Determinant Hosoya Triangle,  
\emph{Integers}, \textbf{21} (2021), 1--24.

\bibitem{BlairRigoAntara} M.~Blair, R.~Fl\'{o}rez,  A.~Mukherjee, and J.L. ~Ramirez, Matrices in the determinant Hosoya triangle, \emph{Fibonacci Quart.} \textbf{58.5} (2020), 34--54. 

\bibitem{Blair} M.~Blair, R.~Fl\'{o}rez, and A.~Mukherjee, Matrices in the Hosoya triangle,
\emph{Fibonacci Quart.} \textbf{57.5} (2019), 15--28.

\bibitem{Ching}  H-Y.~Ching, R.~Fl\'{o}rez, and A.~Mukherjee,  Families of integral cographs within a  triangular arrays, \emph{Spec. Matrices}, \textbf{6} (2020), 257--273.

\bibitem{florezHiguitaMukherjee} R. Fl\'{o}rez, R. Higuita, and A. Mukherjee,
Star of David and other patterns in the Hosoya-like polynomials triangles,  
\emph{J. Integer Seq.} \textbf{21} (2018), Article 18.4.6.

\bibitem{florezHiguitaJunesGCD} R. Fl\'{o}rez, R. Higuita, and L. Junes,
GCD property of the generalized star of David in the generalized Hosoya triangle,
\emph{J. Integer Seq.} \textbf{17} (2014),  Article 14.3.6, 17 pp.
		
\bibitem{florezjunes} R. Fl\'{o}rez and L. Junes,
GCD properties in Hosoya's triangle, \emph{Fibonacci Quart.} \textbf{50} (2012), 163--174.
	
\bibitem{Green} T.~Green and C.~Hamberg, \emph{Pascal's triangle,  second edition},
	CreateSpace Independent Publishing Platform; 2 Csm edition, 2012.
							
\bibitem{hosoya} H. Hosoya, Fibonacci Triangle, \emph{Fibonacci Quart.}  \textbf{14.3} (1976), 173--178.

\bibitem{JohnsonC} R. C. Johnson, Fibonacci numbers and matrices. Unpublished manuscript.
		
\bibitem{koshy} T.~Koshy, \emph{Fibonacci and Lucas Numbers with Applications},
John Wiley, New York, 2001.

\bibitem{sloane} N. J. A. Sloane, The On-Line Encyclopedia of Integer Sequences, \url{http://oeis.org/}.

\bibitem{Vajda} S.~Vajda, \emph{Fibonacci and Lucas numbers, and the golden section. Theory and applications}, John Wiley, New York, 1989.

\bibitem{VanBilliard} J.~VanBilliard, \emph{Pascal's triangle: A study in Combinations}, CreateSpace Independent Publishing Platform, 2014.
\end{thebibliography}
\end{document}